\documentclass[10pt]{article}
	
\usepackage{amsmath}
\usepackage{amsfonts}
\usepackage{amsthm}
\usepackage{mathtools}
\usepackage{enumitem}
\usepackage{hyperref}
\usepackage{cleveref}
\usepackage{microtype}

\newcommand{\R}{\mathbb{R}}
\newcommand{\C}{\mathbb{C}}

\DeclarePairedDelimiterX{\inner}[2]{\langle}{\rangle}{#1, #2}

\DeclareMathOperator{\End}{End}
\DeclareMathOperator{\Rep}{Rep}

\DeclareMathOperator{\SL}{SL}
\DeclareMathOperator{\SO}{SO}
\DeclareMathOperator{\SU}{SU}
\DeclareMathOperator{\PSL}{PSL}
\DeclareMathOperator{\PSp}{PSp}
\DeclareMathOperator{\PSO}{PSO}
\DeclareMathOperator{\G}{G}

\DeclareMathOperator{\id}{id}

\DeclareMathOperator{\im}{im}
\DeclareMathOperator{\Ad}{Ad}

\newcommand{\parder}[2]{\frac{\partial #1}{\partial #2}}

\newcommand{\g}{\mathfrak{g}}
\renewcommand{\k}{\mathfrak{k}}
\newcommand{\p}{\mathfrak{p}}

\newcommand{\delbar}{\overline{\partial}}

\newcommand{\su}{\mathfrak{su}}
\renewcommand{\sl}{\mathfrak{sl}}

\newtheorem{theorem}{Theorem}[section]

\newtheorem{corollary}[theorem]{Corollary}
\newtheorem{lemma}[theorem]{Lemma}
\newtheorem{definition}[theorem]{Definition}

\newtheoremstyle{remarkstyle}
{}
{}
{}
{}
{\bfseries}
{.} 
{7pt}
{}

\theoremstyle{remarkstyle}
\newtheorem{remark}[theorem]{Remark}
\newtheorem*{acknowledgements}{Acknowledgements}
	
\title{Strict plurisubharmonicity of the energy on Teichm\"uller space associated to Hitchin representations}
\author{Ivo Slegers}
	
\begin{document}
\maketitle
	
\begin{abstract}
Let $\Sigma$ be a closed surface of genus at least two and $\rho \colon \pi_1(\Sigma) \to G$ a Hitchin representation into $G=\PSL(n,\R)$, $\PSp(2n,\R)$, $\PSO(n,n+1)$ or $\G_2$. We consider the energy functional $E$ on the Teichm\"uller space of $\Sigma$ which assigns to each point in $\mathcal{T}(\Sigma)$ the energy of the associated $\rho$-equivariant harmonic map. The main result of this paper is that $E$ is strictly plurisubharmonic. As a corollary we obtain an upper bound of $3 \cdot \mathrm{genus}(\Sigma) -3$ on the index of any critical point of the energy functional.
\end{abstract}		

\section{Introduction}\label{sec:introduction_plurisub}
Let $\Sigma$ be a closed surface of genus at least two and let $\rho \colon \pi_1(\Sigma) \to G$ be a Hitchin representation. In this paper we take $G$ to be either $\PSL(n,\R), \PSp(2n,\R)$, $\PSO(n,n+1)$ or the exceptional group $G_2$. Let $K$ be a maximal compact subgroup of $G$. For every complex structure $J$ on $\Sigma$ there exists a (unique) $\rho$-equivariant harmonic map $f_J \colon (\widetilde{\Sigma}, J) \to G/K$. Recall that a map $f \colon \widetilde{\Sigma} \to G/K$ is called $\rho$-equivariant if $f(\gamma x) = \rho(\gamma) f(x)$ for all $\gamma\in \pi_1(\Sigma)$. The energy density of each $f_J$ is $\pi_1(\Sigma)$-invariant. Hence, it descends to $\Sigma$ and can be integrated to obtain the Dirichlet energy of $f_J$. Assigning to a complex structure $J$ the energy of the harmonic map $f_J$ gives us an energy functional on the Teichm\"uller space of $\Sigma$. We will denote this functional by $E \colon \mathcal{T}(\Sigma) \to \R$. The main result of this paper is the following theorem.

\begin{theorem}\label{thm:maintheorem}
Let $G$ be one of the following Lie groups: $\PSL(n,\R)$, $\PSp(2n,\R)$, $\PSO(n,n+1)$ or the exceptional group $\G_2$. If $\rho \colon \pi_1(\Sigma) \to G$ is a Hitchin representation, then the energy functional $E \colon \mathcal{T}(\Sigma) \to \R$ is strictly plurisubharmonic.
\end{theorem}
This theorem extends the results of Tromba (\cite[Theorem 6.2.6]{Tromba}) to a wider class of energy functionals. Tromba considers a fixed hyperbolic metric $g$ on $\Sigma$ and studies the energy functional that assigns to each complex structure $J$ the energy of the harmonic map $(\Sigma,J) \to (\Sigma,g)$ that is homotopic to the identity. He proves that this functional is strictly plurisubharmonic. This corresponds to our result if we take $\rho \colon \pi_1(\Sigma) \to \PSL(2,\R)$ to be a Fuchsian representation. Hitchin representations form a larger class of representations (that contains the Fuchsian representations) so \Cref{thm:maintheorem} can be seen as a natural extension of Tromba's results.

The energy functional $E$ is studied by Labourie in \cite{LabourieCrossRatios}. He proved $E$ is a proper function on Teichm\"uller space and hence has a global minimum. Labourie conjectured that this critical point of the energy functional is unique. This conjecture has been proved in the case that the Lie group $G$ has rank 2 (see \cite{Loftin} and \cite{LabourieCyclic}) but remains open in higher rank. Our result puts a limit on how degenerate a critical point of $E$ can be. More precisely, it implies that the Hessian of $E$ at any critical point is positive definite on a subspace of dimension at least $3\cdot \mathrm{genus}(\Sigma)-3$ (cf. \Cref{cor:indexbound}).

Various examples of plurisubharmonic functions on Teichm\"uller space have been constructed. Notably, in \cite{Yeung} it is proved that Teichm\"uller space admits a bounded and strictly plurisubharmonic exhaustion function. In contrast, the energy functionals we consider in this paper provide interesting examples of strictly plurisubharmonic functions that are proper.

Our proof of \Cref{thm:maintheorem} is based on the work of Toledo in \cite{Toledo}. Our main innovation is the use of Higgs bundles techniques to sharpen the results of that paper in the particular case we consider. Toledo considers a Riemannian manifold $N$ of non-positive Hermitian curvature (see \Cref{sec:toledosresults}) and makes the assumption that for every complex structure $J$ there exists a unique harmonic map $(\Sigma,J) \to N$ in a given homotopy class. He then proves that the functional that assigns to each $J$ the energy of this harmonic map is a plurisubharmonic function on Teichm\"uller space. The setting we consider amounts to taking $N = {\rho(\pi_1(\Sigma))\setminus G/K}$. Our proof of \Cref{thm:maintheorem} combines the result of Toledo with the Higgs bundle description of Hitchin representations to obtain the \textit{strict} plurisubharmonicity of $E$.

We obtain two corollaries to \Cref{thm:maintheorem}. The first, \Cref{cor:indexbound}, gives an upper bound on the index of critical points of $E$. Namely, if $g$ is the genus of $\Sigma$, then the index of a critical point of $E$ is at most $\dim_{\C} \mathcal{T}(\Sigma) = 3g-3$. The second corollary, \Cref{ref:totallyreal}, states that the set of points where $E$ attains its minimal value is locally contained in a totally real submanifold of $\mathcal{T}(\Sigma)$.

The proof of \Cref{thm:maintheorem} and its corollaries will be given in \Cref{sec:proof}. In \Cref{sec:nonabelianhodgecor} we recall the aspects of the Non-Abelian Hodge correspondence and the construction of the Hitchin component that we need for our proof. In \Cref{sec:toledosresults} we describe the results of \cite{Toledo} on which our proof will be based.

\begin{acknowledgements}
The author wishes to thank Professor Ursula Hamenst\"adt for her encouragement and the many helpful suggestions she has made during this project. The author was supported by the IMPRS graduate program of the Max Planck Institute for Mathematics.
\end{acknowledgements}

\section{Non-Abelian Hodge correspondence}\label{sec:nonabelianhodgecor}
We briefly recall the Non-Abelian Hodge correspondence and the construction of the Hitchin component for the case $G = \SL(n,\C)$. We follow parts of the expositions found in \cite{Maubon} and \cite{LiAnIntroduction}. In this section we will denote $G = \SL(n,\C)$ and $K = \SU(n)$. The Lie algebras of these groups we denote by $\g = \sl(n,\C)$ and $\k = \su(n)$ and we let $\p \subset \g$ be the subspace of Hermitian matrices. We have $\g = \k \oplus \p$. Furthermore, let $X$ be a Riemann surface of genus at least two, let $\widetilde{X}$ be its universal cover and denote by $K_X$ the canonical bundle of $X$. Finally, if $E\to X$ is a vector bundle we denote by $\End_0(E)$ the vector bundle of trace free endomorphisms of $E$.

\begin{definition}\label{def:higgsbundle}
A $G$-Higgs bundle over $X$ is a pair $(E,\phi)$ where $E$ is a rank $n$ holomorphic vector bundle over $X$ with trivial determinant bundle and $\phi$ is a holomorphic section of $K_X\otimes \End_0(E)$. We call $(E,\phi)$ stable if any proper sub-G-Higgs bundle has negative degree and we call $(E,\phi)$ polystable if it is a direct sum of stable G-Higgs bundles.
\end{definition}
We denote by $\mathcal{M}_{\text{Higgs}}(G)$ the moduli space of gauge equivalence classes of polystable $G$-Higgs bundles over $X$.  The representation variety $\Rep(\pi_1(X), G)$ is the set of conjugacy classes of reductive representations of $\pi_1(X)$ into $G$. The Non-Abelian Hodge correspondence describes an identification between $\Rep(\pi_1(X),G)$ and $\mathcal{M}_{\text{Higgs}}(G)$.

We first describe how to construct a $G$-Higgs bundle from a representation. Let $\rho \colon \pi_1(X) \to G$ be a reductive representation and consider the $G$-bundle $P_G = (\widetilde{X} \times G)/\pi_1(X) \to X$ where $\pi_1(X)$ acts on the second factor via the representation $\rho$. Let $\omega \in \Omega^1(G,\g)$ be the left Cartan form on $G$. Then the form $\pi^*\omega$ on $\widetilde{X}\times G$ is the connection form of the flat connection of $\widetilde{X}\times G$ (where $\pi \colon \widetilde{X}\times G \to G$ is the projection to the second factor). This form descends to $P_G$ inducing a flat connection on $P_G$ which we will denote by $D$. 

Since $\rho$ is reductive it follows from \cite{Corlette} that there exists a $\rho$-equivariant harmonic map $f \colon \widetilde{X} \to G/K$ (unique up to composition with an element in the centraliser of $\im\rho$). We consider the reduction of the structure group of $P_G$ to $K$ determined by $f$. The projection $G\to G/K$ is a $K$-bundle which we pull back via $f$ to obtain the $K$-subbundle $f^*G \subset \widetilde{X}\times G$. By $\rho$-equivariance of $f$ this bundle descends to a $K$-bundle $P_K \subset P_G$ over $X$. We denote by $\omega^\k$ and $\omega^\p$ the composition of the Cartan form on $G$ with the projections $\g \to \k$ and $\g \to \p$ respectively. Note that on $f^*G$ we have $\pi^*\omega = f^*\omega$ hence $\pi^*\omega = f^*\omega^\k + f^*\omega^\p$. The form $f^*\omega^\k$ descends to $P_K$ and is a connection form. We will denote the connection it determines on $P_K$ by $\nabla$. The form $f^*\omega^\p$ descends to a $\p$-valued one-form on $P_K$. This form is basic and hence determines a section of $T^*X \otimes (P_K \times_{\Ad K} \p)$ which we will call $\Phi$. From the above observations follows that $D = \nabla + \Phi$. 

From the data of $(P_K, \nabla)$ and $\Phi$ a $G$-Higgs bundle can be constructed. We consider $E = P_K \times_K \C^n$ where $K = \SU(n)$ acts on $\C^n$ via the canonical action. The connection $\nabla$ on $P_K$ induces a connection on $E$, that we will also denote by $\nabla$. The $(0,1)$ part of $\nabla$ determines a holomorphic structure on $E$. We note that $\p^\C = \sl(n,\C) = \End_0(\C^n)$ hence ${P_K \times_K \p^\C} = \End_0(E)$. It follows that the $(1,0)$ part of $\Phi$, which we will denote by $\phi = \Phi^{1,0}$, is a section of $K_X \otimes (P_K \times_{\Ad K} \p^\C) = K_X \otimes \End_0(E)$. Finally, we use that the harmonicity condition on the map $f$ translates to $\nabla^{0,1}\phi = \nabla^{0,1}\Phi^{1,0} = 0$. So $\phi$ is a holomorphic section of $K_X\otimes \End_0(E)$. We conclude that the pair $(E,\phi)$ is a $G$-Higgs bundle.

Conversely, if $(E,\phi)$ is a polystable $G$-Higgs bundle, then it follows from a theorem of Hitchin \cite{HitchinSelfDuality} and Simpson \cite{Simpson} that there exists a Hermitian metric $H$ on $E$ such that
\begin{equation*}
F^{\nabla^H} + [\phi, \phi^{*^H}] = 0.
\end{equation*}
Here $\nabla^{H}$ denotes the Chern connection of $H$, $F^{\nabla^H}$ is its curvature and $\phi^{*^H}$ is the adjoint of $\phi$ with respect to $H$. The above condition implies that if we define a connection by setting $D = \nabla^H + \phi + \phi^{*^H}$, then $D$ is flat. We now obtain a representation $\rho \colon \pi_1(X) \to \SL(n,\C)$ by taking a holonomy representation of the flat bundle $E$ around any point $x\in X$.

The Non-Abelian Hodge correspondence states that the two constructions described above are inverses of each other and describe a homeomorphism between $\mathcal{M}_{Higgs}(G)$ and $\Rep(\pi_1(X), G)$.

In the following lemmas we collect two observations about the above construction that we will use in later arguments. By $\rho$-equivariance the bundle $f^*T(G/K)$ defined over $\widetilde{X}$ descends to a bundle over $X$. We will denote this bundle also by $f^*T(G/K)$. We denote by $\nabla^{lc}$ the Levi-Civita connection on $T(G/K)$.
\begin{lemma}\label{lem:identificationvectorbundles}
The bundles $(\End_0(E), \nabla)$ and $(f^*T_\C (G/K), f^*\nabla^{lc})$ are affine isomorphic. That is there is a vector bundle isomorphism $\beta \colon f^*T_\C (G/K) \to \End_0(E)$ with $\beta^*\nabla = f^*\nabla^{lc}$.
\end{lemma}
\begin{proof}
We first observe that $T(G/K) = G\times_{\Ad K} \p$ and hence on $\widetilde{X}$ we have
\begin{equation*}
f^*T(G/K) = f^*(G\times_{\Ad K} \p) = (f^*G) \times_{\Ad K} \p.
\end{equation*}
Both these bundles descend to $X$ so on $X$ we have
\begin{equation*}
f^*T (G/K) = P_K\times_{\Ad K} \p.
\end{equation*}
In the above discussion we saw $P_K\times_{\Ad K} \p^\C = \End_0(E)$ so we find that $f^*T_\C (G/K) = \End_0(E)$. Finally, we observe that $\omega^\k$ on $G$ is the connection form that induces the Levi-Civita connection on $G\times_{\Ad K}\p$. So $f^*\omega^\k$ induces the connection $f^*\nabla^{lc}$ on $f^*T_\C (G/K)$ and also, by construction, induces the connection $\nabla$ on $\End_0(E)$. We conclude that the two bundles are indeed affine isomorphic.
\end{proof}
\begin{lemma}\label{lem:identificationderivative}
Consider the derivative of the map $f$ as a section $df \in T^*X \otimes f^*T(G/K)$. Then under the above described correspondence of vector bundles we have the following equality of $P_K\times_{\Ad K} \p$ valued one-forms
\begin{equation*}
\beta(df) = \Phi
\end{equation*}
As a consequence we obtain, if we denote $d'f = (df)^{1,0}$, that
\begin{equation*}
\beta(d'f) = \phi.
\end{equation*}
\end{lemma}
\begin{proof}
We consider the vector bundle valued one-form $\Psi\in T^*X \otimes (P_K \times_{\Ad K} \p)$ defined by $\Psi = \beta(df)$. We lift $\Psi$ first to $P_K$ and then to $f^*G$ to obtain a $\p$-valued one-form $\widetilde{\Psi}$ on $f^*G$. Let $p \colon G \to G/K$ be the quotient map. By unrolling the definition of $\beta$ we can describe $\widetilde{\Psi}$ as follows. Let $(x,g) \in f^*G$ (i.e. $f(x) = p(g)$) and $(X,A) \in T_{(x,g)}f^*G$. Then $\widetilde{\Psi}((X,A)) = \xi$ where $\xi \in \p$ is the unique element such that $g_* dp(\xi) = df(X)$. We now consider the form $f^*\omega^\p$. Here $f \colon f^*G \to G$ is the map induced by the pull back construction and is given by $f(x,g) = g$. We have $(f^*\omega^\p)((X,A)) = \omega^\p(A)$. The condition $(X,A) \in T_{(x,g)}f^*G$ implies $df(X) = dp(A)$ hence we observe that
\begin{equation*}
g_*dp(f^*\omega^\p((X,A))) = g_* dp(\omega^\p(A)) = dp(A) = df(X).
\end{equation*}
We find that $f^*\omega^\p((X,A)) = \xi$ and hence $f^*\omega^\p = \widetilde{\Psi}$ on $f^*G$. Since $f^*\omega^\p$ descends to $\Phi$ and $\widetilde{\Psi}$ descends to $\Psi = \beta(df)$ we conclude that indeed $\beta(df) = \Phi$.
\end{proof}

\subsection{Hitchin component}
If $(p_2,\hdots, p_n)$ is a basis for the space of conjugation invariant polynomials on $\sl(n,\C)$ we can construct a map
\begin{equation*}
p \colon \mathcal{M}_{\text{Higgs}}(G) \to \oplus_{i=2}^n H^0(X;K_X^i) \colon (E,\phi) \mapsto (p_2(\phi),\hdots,p_n(\phi))
\end{equation*}
via the Chern-Weil construction. This map is called the Hitchin fibration. A section of this map can be constructed as follows. Let $K_X^{1/2}$ be a choice of holomorphic line bundle over $X$ that squares to $K_X$. We set
\begin{equation*}
E = K_X^{\frac{n-1}{2}} \oplus K_X^{\frac{n-3}{2}} \oplus \dots \oplus K_X^{\frac{3-n}{2}} \oplus K_X^{\frac{1-n}{2}}.
\end{equation*}
Then $K_X\otimes \End_0(E) \subset \oplus_{i,j=1}^n K_X^{i-j+1}$. For $(q_2,\hdots,q_n) \in \oplus_{i=2}^n H^0(X; K^i_K)$ we define
\begin{align}\label{eq:hitchinsection}
s(q_2,\hdots,q_n) = 
\left(
E, ~\phi = 
\begin{pmatrix}
0 & q_2 & q_3 & \hdots & q_n\\
r_1 & 0 & q_2 & \hdots & q_{n-1}\\
0 & r_2 & 0 & \ddots & \vdots\\
\vdots & \ddots & \ddots & \ddots & q_2 \\
0 &  \hdots & 0 & r_{n-1} & 0
\end{pmatrix}
\right)
\end{align}
where $r_i = \frac{i(n-i)}{2}$. For a suitable choice of $(p_2,\hdots,p_n)$ we have that $s$ is indeed a section of $p$. Hitchin proved in \cite{HitchinLieGroups} that representations determined (via the Non-Abelian Hodge correspondence) by Higgs bundles in the image of this section take values in $\SL(n,\R)$. Furthermore, these representations constitute precisely a connected component of the space $\Rep(\pi_1(X), \SL(n,\R))$. We call this connected component the \textit{Hitchin component} and representations contained in it \textit{Hitchin representations}. We note that the exact form of the section $\phi$ in \Cref{eq:hitchinsection} depends on a choice of irreducible embedding of $\SL(2,\R)$ into $\SL(n,\R)$. The resulting sections for different choices can be related by a gauge transformation. We follow the choice made in \cite{LiAnIntroduction} and hence $\phi$ differs slightly from the section that appears in \cite{HitchinLieGroups}.

By composing with the projection $\SL(n,\R) \to \PSL(n,\R)$ a Hitchin representation induces a representation into $\PSL(n,\R)$. Hitchin proved (\cite[Section 10]{HitchinLieGroups}) that $\Rep(\pi_1(\Sigma), \PSL(n,\R))$ contains a connected component consisting entirely of representations that are obtained in this way (i.e. each can be lifted to a Hitchin representations into $\SL(n,\R)$). We call representations of $\pi_1(\Sigma)$ into $\PSL(n,\R)$ that lie in this component also \textit{Hitchin representations}.

If $G^r$ is an adjoint group of the split real form of a complex simple Lie group  it is also possible to identify a Hitchin component in $\Rep(\pi_1(\Sigma),G^r)$ using a Higgs bundle argument (\cite{HitchinLieGroups}). However, for convenience we give an alternative definition. Namely, for such $G^r$ there exists an irreducible representation $\iota_{G^r} \colon \PSL(2,\R) \to G^r$ that is unique up to conjugation. Composing a Fuchsian representation $\rho_0 \colon \pi_1(\Sigma) \to \PSL(2,\R)$ that corresponds to a point in Teichm\"uller space with $\iota_{G^r}$ yields a representation into $G^r$. The Hitchin component of $\Rep(\pi_1(\Sigma),G^r)$ can be defined as the connected component containing $\iota_{G^r} \circ \rho_0$.

The cases $G^r = \PSp(2n,\R)$, $\PSO(n,n+1)$ or $\G_2$ have the special feature that if we consider $G^r$ as a subset of $\PSL(m,\R)$ (for $m = 2n, 2n+1$ or $7$ respectively), then $\iota_{G^r} = \iota_{\PSL(m,\R)}$. Hence, the Hitchin component for $G^r$ can be realised as a subset of the Hitchin component for $\PSL(m,\R)$.

\section{Plurisubharmonicity}\label{sec:toledosresults}
In this section we explain some of the results of \cite{Toledo} and introduce some notation used in that paper that we will also use. Let 
\begin{equation*}
\mathcal{C} = \{J \in C^\infty(\Sigma, \End(T\Sigma)) \mid J^2 = -\id\}
\end{equation*}
be the set of almost complex structures on $\Sigma$. We let $N$ be a Riemannian manifold of non-positive Hermitian sectional curvature. This condition means that $R(X,Y, \overline{X}, \overline{Y}) \leq 0$ for all $X,Y \in TN\otimes \C$ where $R$ is the complex multilinear extension of the Riemannian curvature tensor of $N$. 

If $A$ is an endomorphism of $T\Sigma$, then for any one-form $\alpha \in \Omega^1(\Sigma)$ we denote $A\alpha = -\alpha\circ A$. In particular, if $J\in \mathcal{C}$, then $Jdf = -df \circ J = df\circ J^{-1}$. For any $J\in \mathcal{C}$ the Dirichlet energy of a map $f \colon \Sigma \to N$ is given by
\begin{equation*}
\mathcal{E}(J, f) = \frac{1}{2} \int_{\Sigma} \left\langle df \wedge J df \right\rangle.
\end{equation*}
A map is harmonic if it is a critical point of this functional. We fix a homotopy class of maps $\Sigma \to N$. We make the assumption that for each $J \in \mathcal{C}$ there exists a unique harmonic map $f_J \colon (\Sigma,J) \to N$ in this homotopy class. We assume further that the maps $f_J$ depend smoothly on $J$. These assumptions will be satisfied in the situation we will consider. Define $E \colon \mathcal{C} \to \R$ by $E(J) = \mathcal{E}(J, f_J)$. This map descends to Teichm\"uller space because if $\phi \in \mathrm{Diff}_0(\Sigma)$, then $f_{\phi^*J} = \phi^*f_J$ hence $E(\phi^*J) = \mathcal{E}(\phi^*J, \phi^*f_J) = \mathcal{E}(J,f_J) = E(J)$. The main result of \cite{Toledo} is that $E$ is a plurisubharmonic function on Teichm\"uller space. 

To state this result formally we consider a small disk $D \subset \C$ centred around $0$ and a holomorphic family of complex structures $J \colon D \to \mathcal{C}$. Denote by $u = {s+ it}$ the complex coordinates on $D$. Set $E(s,t) = E(J(s,t))$ and $f(s,t) = f_{J(s,t)}$. We define
\begin{equation*}
W = \parder{f}{s} + i \parder{f}{t} \in \Gamma^\infty(f^*T_\C N).
\end{equation*}
We equip $\Sigma$ with the complex structure $J_0 = J(0,0)$ and denote by $T_\C \Sigma = T_{1,0}\Sigma\oplus T_{0,1}\Sigma$ and $T^*_\C \Sigma = T^{1,0}\Sigma \oplus T^{0,1}\Sigma$ the induced splittings of the tangent and cotangent space into $+i$ and $-i$ eigenspaces of $J_0$. The complexification of the derivative $df$ splits into a $(1,0)$ and $(0,1)$ part denoted by $d'f$ and $d''f$ respectively. Similarly, if $s$ is a section of a vector bundle equipped with a connection $\nabla$, then we denote by $d_\nabla's$ and $d_\nabla''s$ respectively the $(1,0)$ and $(0,1)$ part of $\nabla s$. Finally, we consider
\begin{equation*}
H = \parder{J}{s}(0,0) \in T_{J_0} \mathcal{C}.
\end{equation*}
The endomorphism $H$ of $TM$ anti-commutes with $J_0$ hence its complexification can be written as $H = \mu + \overline{\mu}$ with $\mu$ a smooth section of $T^{0,1}\Sigma\otimes T_{1,0}\Sigma$.

Theorem 2 of \cite{Toledo} now states:
\begin{theorem}\label{thm:toledoplurisubharmonicity}
We have
\begin{equation*}
\Delta E(0,0) \geq 0
\end{equation*}
and in case of equality we have 
\begin{align}\label{eq:equality}
d_\nabla''W = \pm \mu d'f.
\end{align}
\end{theorem}
The last statement in this theorem is not explicitly stated in \cite{Toledo} but follows from the arguments used to prove the first statement. We briefly clarify how \Cref{eq:equality} is obtained when $\Delta E(0,0) = 0$ (see also the proof of \cite[Theorem 3]{Toledo}). In this section any reference to a numbered equation will to refer to an equation in \cite{Toledo}.

Toledo first calculates (Equation 16) that $\Delta E(0,0) = -a + b$ where
\begin{equation*}
a= -\int_{\Sigma} \left\langle d_\nabla \parder{f}{s} \wedge H df \right \rangle + \left\langle d_\nabla \parder{f}{t} \wedge J_0 H df \right \rangle \text{ and } b = \int_\Sigma \left\langle df \wedge J_0 H^2 df \right\rangle.
\end{equation*}
We denote also
\begin{equation*}
\alpha = \int_\Sigma \left\langle d'_\nabla \overline{W} \wedge J_0 d''_\nabla W\right \rangle \text{ and } \rho = \int_\Sigma R\left(\parder{f}{z}, W, \overline{\parder{f}{z}}, \overline{W} \right) dx \wedge dy.
\end{equation*}
Inequality (26) yields that $a \leq \alpha + \frac{b}{2}$ and Equation (29) gives that $\alpha = \frac{a}{2} + \rho$. Putting these together gives $\alpha \leq \frac{1}{2}(a+b) + 2\rho$ or equivalently $a \leq b + 4\rho$ (which is Inequality (30)). The non-positive Hermitian curvature condition implies $\rho\leq 0$ hence $a \leq b$ from which follows that $\Delta E(0,0) \geq 0$.

If the family $J$ is such that $\Delta E(0,0) = 0$, then we see that equality holds in inequalities (26) and (30). The remarks made by Toledo after Inequality (26) tell us that Inequality (26) is an equality if and only if $d''_\nabla W = \pm \mu d'f$.

We note that in this case we also have $\rho = 0$ which means $R\left(\parder{f}{z}, W, \overline{\parder{f}{z}}, \overline{W} \right) = 0$ everywhere. However, we do not use this in our proof.

\begin{remark}\label{rmk:noncompactness}
We note that in the statements of Theorems 1 and 2 in \cite{Toledo} the manifold $N$ is assumed to be compact. This is something that will not be true in the application we have in mind. The compactness assumption is used to guarantee the existence of a harmonic map $(\Sigma,J) \to N$ in a given homotopy class for every $J$. This is not necessarily true when $N$ is not compact. However, in the situation we consider the existence of such harmonic maps follows from the results of Corlette (\cite{Corlette}). An inspection of the proof in \cite{Toledo} shows that the compactness of $N$ plays no further role. This means we are free to apply \Cref{thm:toledoplurisubharmonicity} even if $N$ is not compact, as long as the existence of a (unique) harmonic map for each $J \in \mathcal{C}$ is guaranteed.
\end{remark}

\section{Proof}\label{sec:proof}
We turn now to the proof of \Cref{thm:maintheorem}. We observe first that it is enough to give a proof for $G = \PSL(n,\R)$. Namely, if $G$ equals $\PSp(2n,\R)$, $\PSO(n,n+1)$ or $\G_2$, then the inclusion $G \subset \PSL(m,\R)$ (for $m = 2n, 2n+1$ or $7$ respectively) induces an inclusion of the Hitchin component for $G$ into the Hitchin component for $\PSL(m,\R)$. Moreover, via the totally geodesic embedding $G/K \subset \PSL(m,\R)/\PSO(m)$ a harmonic map $\widetilde{X} \to G/K$ equivariant for a representation $\rho \colon \pi_1(X) \to G$ can be seen as a harmonic map into $\PSL(m,\R)/\PSO(n)$ equivariant for $\rho$ as a representation into $\PSL(m,\R)$. In particular, the energy functional $E$ is unchanged if we view $\rho$ as a representation into $\PSL(m,\R)$ rather then into $G$.

We consider now the energy function associated to a $\PSL(n,\R)$-Hitchin representation. We lift this representation to a representation into $\SL(n,\R)$ which we denote by $\rho \colon \pi_1(\Sigma) \to \SL(n,\R)$. From now on we denote $G = \SL(n,\R)$ and $K = \SO(n)$. Hitchin representations act freely and properly on $G/K$ (\cite{LabourieAnosovFlows}) so we can consider the locally symmetric space $N = \rho(\pi_1(\Sigma)) \setminus G / K$. The representation $\rho$ determines a homotopy class of maps $\Sigma \to N$ that lift to $\rho$-equivariant maps $\widetilde{\Sigma} \to G/K$. Equivariant harmonic maps for Hitchin representations are unique and depend smoothly on $J$ (see \cite{EellsLemaire} or \cite{Slegers}). Hence, we can consider the energy functional $E \colon \mathcal{T}(\Sigma) \to \R$ as defined in \Cref{sec:toledosresults}. We note that this coincides with the energy functional as described in \Cref{sec:introduction_plurisub}. In \cite{SampsonKahlerGeometry} Sampson proved that locally symmetric spaces of non-compact type have non-positive Hermitian sectional curvature. So \Cref{thm:toledoplurisubharmonicity} applies to $E$.

We now give a proof of \Cref{thm:maintheorem}. The strategy is similar to the proof of \cite[Theorem 3]{Toledo} in which strict plurisubharmonicity is proved when the target is assumed to have strictly negative Hermitian sectional curvature. It is interesting that this strictly negative curvature condition can be replaced by the explicit information about the form of the harmonic map that is provided by the Higgs bundle picture.

\begin{proof}[Proof of \Cref{thm:maintheorem}]
We use the notation introduced in \Cref{sec:toledosresults}. Suppose that $J \colon D \to \mathcal{C}$ is a holomorphic family of complex structures such that $\Delta E(0,0) = 0$. It then follows from \Cref{thm:toledoplurisubharmonicity} that \Cref{eq:equality} holds.

We note that $W$ is a smooth section of $f^*T_\C N$. Using \Cref{lem:identificationvectorbundles} we can view it as a section of $\End_0(E)$ by considering $\nu = \beta(W)$. Since $\beta$ is an affine isomorphism we have $d''_\nabla \nu = \beta(d''_\nabla W)$. Taking into account \Cref{lem:identificationderivative} we see that \Cref{eq:equality} is equivalent to
\begin{gather}\label{eq:equality2}
d''_\nabla \nu = \pm \mu \phi.
\end{gather}
We write $\nu = (\nu_{i,j})_{i,j}$ with each $\nu_{i.j}$ a smooth section of $K^{j-i}$. Keeping in mind the expression for $\phi$ as given in \Cref{eq:hitchinsection} we consider now the (2,1) component of the matrices on both sides of \Cref{eq:equality2}. This gives
\begin{equation*}
\delbar \nu_{2,1} = \pm \mu (r_1 \cdot 1) = \pm \frac{1}{2} \mu.
\end{equation*}
Here $\nu_{2,1}$ is a section of $K^{-1} = T_{1,0} \Sigma$. The above equality implies that $[\mu] = 0$ in $H^1(X, T_{1,0}\Sigma)$ which means precisely that the tangent vector $H \in T_{J_0} \mathcal{C}$ projects to zero in $T_{[J_0]} \mathcal{T}(\Sigma)$.

We conclude that for any family $J$ of complex structures inducing a non-zero tangent vector in Teichm\"uller space we have $\Delta E(0,0)>0$. This concludes the proof.
\end{proof}
As a first corollary of \Cref{thm:maintheorem} we obtain a bound on the index of the critical points of $E$. We recall that if $g = \mathrm{genus}(\Sigma)$, then $\dim_{\R} \mathcal{T}(\Sigma) = 6g - 6$.

\begin{corollary}\label{cor:indexbound}
Under the assumptions of \Cref{thm:maintheorem} the index of a critical point of $E$ is at most $\dim_{\C} \mathcal{T} = 3g -3$.
\end{corollary}

\begin{proof}
Assume $[J]\in \mathcal{T}(\Sigma)$ is a critical point of $E$. Let $H$ be the Hessian of $E$ at this point and denote by $\widetilde{H}$ its sesquilinear extension of the complexified tangent space of $\mathcal{T}(\Sigma)$. The forms $H$ and $\widetilde{H}$ have the same index. If $(z^1,\hdots,z^{3g-3})$ are complex coordinates around $[J]$, then the strict plurisubharmonicity property of $E$ implies that
\begin{equation*}
\widetilde{H}(u, v) = \parder{^2 E}{z^\alpha \partial \overline{z}^\beta} u^\alpha \overline{v^\beta}
\end{equation*}
is positive definite. This means that $\widetilde{H}$ is positive definite on the subspace of dimension $3g-3$ that is spanned by the vectors $\parder{}{z^\alpha}$ and as a result has index at most $3g-3$.
\end{proof}

Finally, we obtain the following corollary by applying the results of \cite{HarveyWells} to the function $f = E - \min_{[J]\in\mathcal{T}(\Sigma)} E([J])$. We call a submanifold $P$ of $\mathcal{T}(\Sigma)$ \textit{totally real} if $T_p P$ contains no non-zero complex subspaces of $T_p\mathcal{T}(\Sigma)$ for all $p \in P$.

\begin{corollary}\label{ref:totallyreal}
The set
\begin{equation*}
M=\{[J]\in \mathcal{T}(\Sigma) \mid E \text{ attains its global minimum at } [J]\}.
\end{equation*}
is locally contained in totally real submanifolds of $\mathcal{T}(\Sigma)$. More precisely for every $[J] \in M$ there exists an open neighbourhood $U\subset \mathcal{T}(\Sigma)$ of $[J]$ and a totally real submanifold $P \subset U$ such that $M \cap U \subset P$. In particular, at smooth points of $M$ its tangent space is totally real. It follows that the Hausdorff dimension of $M$ is at most $3g-3$.
\end{corollary}

\bibliographystyle{alpha}
\bibliography{bibliography}

\end{document}